\documentclass[12pt]{amsart}

\usepackage{amssymb,latexsym}

\usepackage{enumerate}

\usepackage[T1]{fontenc}

 \usepackage[french,english]{babel}

\makeatletter

\@namedef{subjclassname@2010}{

  \textup{2010} Mathematics Subject Classification}

\makeatother
\newtheorem{thm}{Theorem}[section]
\newtheorem{Con}[thm]{Conjecture}

\newtheorem{cor}[thm]{Corollary}
\newtheorem{lem}[thm]{Lemma}
\newtheorem{pro}[thm]{Proposition}
\theoremstyle{definition}

\newtheorem{rem}[thm]{Remark}

\numberwithin{equation}{section}

\newcommand{\ex}{\mathbb{E}}
\newcommand{\re}{\textup{Re}}

\newcommand{\ep}{\varepsilon}

\newcommand{\newabstract}[1]{%
  \par\bigskip
  \csname otherlanguage*\endcsname{#1}%
  \csname captions#1\endcsname
  \item[\hskip\labelsep\scshape\abstractname.]
}

\begin{document}

\baselineskip=17pt

\title[An effective LI conjecture and applications]{An effective Linear Independence Conjecture for the zeros of the Riemann zeta function and applications}

\author{Youness Lamzouri}

\dedicatory{For Helmut Maier on his seventieth birthday}

\address{
Institut \'Elie Cartan de Lorraine, 
Universit\'e de Lorraine, et Institut Universitaire de France,
BP 70239, 
54506 Vandoeuvre-l\`es-Nancy Cedex,
France}

\email{youness.lamzouri@univ-lorraine.fr}


\begin{abstract} Using the same heuristic argument leading to the Lang-Waldschmidt Conjecture in the theory of linear forms in logarithms, we formulate an effective version of the Linear Independence conjecture for the ordinates of the non-trivial zeros of the Riemann zeta function. Then assuming this conjecture, we obtain Omega results for the error term in the prime number theorem, which are conjectured to be best possible by  Montgomery. This was claimed to appear in Monach's thesis by Montgomery and Vaughan in their classical book, but such a result or its proof is nowhere to be found in this thesis or in the litterature. Moreover, if in addition to this effective Linear Independence conjecture we assume a conjecture of Gonek and Hejhal on the negative moments of the derivative of the Riemann zeta function, we can show that the summatory function of the M\"obius function $M(x)$ is $\Omega_{\pm}\left(\sqrt{x}(\log\log\log x)^{5/4}\right)$. This conditionally resolves a part of a conjecture of Gonek. 

\end{abstract}


\thanks{}

\maketitle


\section{Introduction}

The Linear Independence conjecture LI states that the positive imaginary parts of the zeros of the Riemann zeta function $\zeta(s)$ are linearly independent over the rationals. This conjecture first appeared in a work of Wintner \cite{Wi}, who used it to study the error term in the prime number theorem. Let 
$$ E(x):= \frac{\psi(x)-x}{\sqrt{x}},$$
where $$\psi(x)= \sum_{n\leq x} \Lambda(n),$$
and $\Lambda(n)$ is the classical von Mongoldt function, which equals $\log p$ if $n$ is a power of the prime $p$, and equals $0$ otherwise. Assuming the Riemann hypothesis RH and the LI conjecture, Wintner \cite{Wi} showed\footnote{Wintner showed that the existence of the measure $\nu$ can be proved under RH alone, but needed LI to prove that it is absolutely continuous.} that there exists an absolutely continuous probability measure $\nu$ such that 
\begin{equation}\label{EqLimDistribError} \lim_{X\to \infty} \int_{2}^X f(E(x)) \frac{dx}{x}= \int_{-\infty}^{\infty} f(x) d\nu(x),
\end{equation}
for all bounded continuous functions $f$ on $\mathbb{R}.$

A few years later, Ingham \cite{In} used the LI conjecture to disprove Mertens' conjecture for the summatory function of the M\"obius function, which asserts that $|M(x)|\leq \sqrt{x}$ for all  $x>1$, where 
$$ M(x):=\sum_{n\leq x} \mu(n).$$
This conjecture was unconditionally disproved by Odlyzko and te Riele \cite{OdRi} in 1985. Assuming LI, Ingham \cite{In} proved in 1942 a stronger version of this statement, namely that 
$$\limsup_{x\to \infty}\frac{M(x)}{\sqrt{x}}=\infty, \text{ and } \liminf_{x\to \infty}\frac{M(x)}{\sqrt{x}}=-\infty. $$

More recently, the LI conjecture was extended to families of Dirichlet $L$-functions attached to primitive characters modulo $q$ by Rubinstein and Sarnak \cite{RuSa}, who used it to study Chebyshev's bias for primes in arithmetic progressions. Indeed, they were the first to provide a rigorous justification for the existence of this bias, assuming the LI conjecture in this context together with the generalised Riemann hypothesis. More precisely, they proved under these assumptions that  ``the event'' $\pi(x; q, a) >\pi(x; q, b)$ (where $\pi(x; q, a)$ denotes the number of primes $p\leq x$ such that $p\equiv a \pmod q$) has a logarithmic density exceeding $1/2$, if and only if $a$ is a quadratic non-residue, and $b$ is a quadratic residue modulo $q$. Pushing these ideas further, several authors used the LI conjecture to study the logarithmic densities in Chebyshev's bias and more general prime number races (see for example \cite{FeMa}, \cite{FiMa}, \cite{FHL}, \cite{HaLa}, \cite{La}, \cite{Me}).

In this paper, we will formulate an effective version of the LI conjecture for the zeros of the Riemann zeta function, and use it to exhibit extreme values of the error term in the prime number theorem, the summatory function of the M\"obius function, as well as more general sums over the zeros of the zeta function. 

Following the same heuristic argument leading to the Lang-Waldschmidt Conjecture (1978) in the theory of linear forms in logarithms \cite{Lang} (see the introduction to Chap. X and XI, p. 212-217), we formulate the following conjecture. 

\begin{Con}[The Effective LI conjecture (ELI)]\label{EffectiveLI}
Let $\{\gamma\}$ denote the set of the ordinates of the non-trivial zeros of the Riemann zeta function. For every $\ep>0$ there exists a positive constant $C_{\ep}$ (possibly not effective) such that for all real numbers $T\geq 2$ we have 
\begin{equation}\label{ELI}\Big|\sum_{0<\gamma\leq T} \ell_{\gamma} \gamma \Big|\geq C_{\ep} e^{-T^{1+\ep}},
\end{equation}
where the $\ell_{\gamma}$ are integers, not all zero, such that $|\ell_{\gamma}|\leq N(T),$ and $N(T)$ is the number of zeros of $\zeta(s)$ with imaginary part in $[0, T]$.
\end{Con}

The heuristic argument employed by Lang and Waldschmidt works in our setting as follows.  Assuming the LI conjecture, the linear combinations $\sum_{0<\gamma\leq T} \ell_{\gamma} \gamma$ with $|\ell_{\gamma}|\leq N(T)$ are all distinct and there are $(2N(T))^{N(T)}= \exp\big((1/(2\pi)+o(1)) T(\log T)^2\big)$ of them. Furthermore, they all lie in an interval of the form $[-C T^3(\log T)^2, CT^3(\log T)^2]$ for some positive constant $C$, since $|\sum_{0<\gamma\leq T} \ell_{\gamma} \gamma|\leq N(T) \sum_{0<\gamma\leq T} \gamma \ll T^3(\log T)^2.$ 
If these sums were equidistributed in this interval, one would have 
\begin{align*}
\left|\sum_{0<\gamma\leq T} \ell_{\gamma} \gamma\right|&\gg T^3(\log T)^2 \exp\left(-\left(\frac{1}{2\pi}+o(1)\right) T(\log T)^2\right)\\
&=  \exp\left(-\left(\frac{1}{2\pi}+o(1)\right) T(\log T)^2\right), 
\end{align*}
for all $|\ell_{\gamma}|\leq N(T)$ such that the $\ell_{\gamma}$'s are not all $0$.
On the other hand, by considering all such linear combinations with coefficients in the interval $[0, N(T)]$, it follows from  Dirichlet's box principle that there is a linear combination $\sum_{0<\gamma\leq T} \ell_{\gamma} \gamma$ with the $\ell_{\gamma}$'s all in the interval $[-N(T), N(T)]$ and such that 
$$ |\sum_{0<\gamma\leq T} \ell_{\gamma} \gamma|\ll \exp\left(\left(-\frac{1}{2\pi}+o(1)\right) T(\log T)^2\right).
$$
Hence, one might guess that for every fixed $\ep$ there exists a positive constant $C_{\ep}$ (which is possibly ineffective), such that 
\begin{equation}\label{Strong ELI}|\sum_{0<\gamma\leq T} \ell_{\gamma} \gamma|\geq C_{\ep} \exp\left(-\frac{(1+\ep)}{2\pi}T(\log T)^{2}\right),
\end{equation}
for all $|\ell_{\gamma}|\leq N(T)$ such that the $\ell_{\gamma}$'s are not all $0$. Note that this hypothesis is much stronger than the Effective LI conjecture stated in \eqref{ELI}.

\subsection{The error term in the prime number theorem} 

A classical result of Von Koch (see \cite[Page 116]{Da}) asserts that the Riemann hypothesis implies the  asymptotic formula 
$$ \psi(x)= x +O\big(\sqrt{x} (\log x)^2\big).$$
Under the same assumption, Gallagher \cite{Ga} showed that one can improve this to
\begin{equation}\label{EqGallagher}
\psi(x)= x +O\big(\sqrt{x} (\log \log x)^2\big),
\end{equation}
except on a set of finite logarithmic measure. 
On the other hand, Littlewood (see \cite[Chapter 5]{InBook}) proved unconditionally that 
\begin{equation}\label{EqLittlewood}
  \psi(x)-x= \Omega_{\pm}(\sqrt{x} \log\log\log x).  
\end{equation}

Here the notation $f(x)=\Omega_{+}(g(x))$ means $\limsup_{x\to \infty} f(x)/g(x)>0$, and $f(x)=\Omega_{-}(g(x))$ means $\liminf_{x\to \infty} f(x)/g(x)<0$. Moreover, $f(x)=\Omega_{\pm}(g(x))$ means that we both have $f(x)=\Omega_{+}(g(x))$ and $f(x)=\Omega_{-}(g(x))$.
In proving the Omega result \eqref{EqLittlewood}, Littlewood assumed RH, since otherwise one has a stronger estimate 
\begin{equation}\label{EqRHFALSEPNT}
    \psi(x)-x=\Omega_{\pm}\left(x^{\sigma-\ep}\right),
\end{equation}
if $\rho=\sigma+it$ is a non-trivial zero of the zeta function with $\sigma>1/2$.

By studying the tail of the  limiting distribution $\nu$ in \eqref{EqLimDistribError}, Montgomery \cite{Mo} conjectured in 1980 that 
\begin{equation}\label{ConjectureLim}
 \limsup_{x\to \infty} \frac{\psi(x)-x}{\sqrt{x}(\log\log\log x)^2}= \frac{1}{2\pi}, \text{ and } \liminf_{x\to \infty} \frac{\psi(x)-x}{\sqrt{x}(\log\log\log x)^2}= -\frac{1}{2\pi}.
 \end{equation}
 Our first result resolves a part of this conjecture under ELI. 
\begin{thm}\label{ThmErrorPNT}
Assume the Effective LI conjecture (ELI). Then we have
\begin{equation}\label{LimsupPNT}
 \limsup_{x\to \infty} \frac{\psi(x)-x}{\sqrt{x}(\log\log\log x)^2}\geq \frac{1}{2\pi},
 \end{equation}
 and 
 \begin{equation}\label{LiminfPNT}
 \liminf_{x\to \infty} \frac{\psi(x)-x}{\sqrt{x}(\log\log\log x)^2}\leq -\frac{1}{2\pi},
 \end{equation}
\end{thm}

\begin{rem}\label{MontgomeryMonach}
In a remark at the bottom of  page 483 of the book of Montgomery and Vaughan \cite{MoVa}, it is claimed that Montgomery and Monach proved Theorem \ref{ThmErrorPNT} above in Monach's thesis \cite{Monach}, under a slightly weaker assumption of the Effective LI conjecture (where the coefficients $\ell_{\gamma}$ are taken to be bounded). However, the author has examined Monach's thesis with great scrutiny, and did not find any such proof or ideas that can lead to a proof of such a result. Furthermore, our proof of Theorem \ref{ThmErrorPNT} relies on the ELI conjecture with coefficients up to $N(T)$, in order to show that the contribution of the small zeros in the explicit formula \eqref{ExplicitFormulaPsi} below can be very large with ``high probability". This is achieved by estimating large moments of such sums. In particular, restricting the coefficients to a bounded set in the ELI conjecture will only allow one to estimate bounded moments of such sums, and this will not be enough for our method to work.
\end{rem}


\subsection{The summatory function of the M\"obius function}
An equivalent formulation of the Riemann hypothesis is the bound 
\begin{equation}\label{EqMobiusRH}
M(x)\ll _{\ep}x^{1/2+\ep},
\end{equation}
for every fixed $\ep>0$.
Assuming RH, Landau showed that one can take $\ep=\log\log\log x/\log\log x$ in \eqref{EqMobiusRH}, and this was improved to $\ep=1/\log\log x$ by Titchmarsh \cite[Page 371]{Ti}. The best known bound for $M(x)$ under RH is due to Soundararajan \cite{So} who showed that 
$$M(x)\ll \sqrt{x} \exp\left((\log x)^{1/2+\ep}\right).
$$
This improved an earlier result of Maier and Montgomery \cite{MaMo} who obtained the exponent $39/61$ instead of $1/2$ in this result. 
On the other hand, the best unconditional omega result for $M(x)$ is 
\begin{equation}\label{EqOmegaM}
M(x)= \Omega_{\pm}(\sqrt{x}).
\end{equation}
Gonek (unpublished, see Ng \cite{Ng}) conjectured that there exists a positive constant $B$ such that 
\begin{equation}\label{EqConjectureGonekM}
\limsup_{x\to \infty} \frac{M(x)}{\sqrt{x} (\log\log \log x)^{5/4}}= B, \text { and } \liminf_{x\to \infty} \frac{M(x)}{\sqrt{x} (\log\log \log x)^{5/4}}= -B.
\end{equation}
In 2012 Ng (unpublished)\footnote{Private communication.} worked out the value of the constant $B$ 
and conjectured that \eqref{EqConjectureGonekM} holds with 
$$B = \frac{8}{5\sqrt{\pi}} e^{3\zeta'(-1)-\frac{11}{12}\log(2)} \prod_p\left(\left(1-p^{-1}\right)^{1/4}\sum_{k=0}^{\infty}\left(\frac{\Gamma(k-1/2)}{k!\Gamma(-1/2)}\right)^2p^{-k}\right)=0.26739....$$
Since our goal is to improve the omega result \eqref{EqOmegaM}, we might assume the Riemann hypothesis and the simplicity of the zeros of the zeta function\footnote{Currently we know that more than $2/5$ of the zeros of the zeta function are simple (see \cite{BCY}), and this proportion can be increased to $19/27$ assuming RH (see \cite{BuHe}).}, an assumption which is widely believed.  Indeed if RH is false and $\rho=\sigma+it$ is a non-trivial zero of $\zeta(s)$ with $\sigma>1/2,$ then one has 
\begin{equation}\label{EqOmegaRHFalseM}
M(x)=\Omega_{\pm}\left(x^{\sigma-\ep}\right),
\end{equation}
for any fixed $\ep$ (see \cite{InBook}). Moreover, if $\rho$ is a zero with multiplicity $m\geq 2$ then we have 
\begin{equation}\label{EqOmegaMultiplicity}
M(x)= \Omega_{\pm}\left(\sqrt{x}(\log x)^{m-1}\right).
\end{equation}
In particular, if RH were to be false, or $\zeta(s)$ happens to have a multiple zero, then we would trivially have
\begin{equation}\label{EqOmegaPredictedM}
M(x)=\Omega_{\pm} \left(\sqrt{x}(\log\log\log x)^{5/4}\right),
\end{equation}
which is predicted by Gonek's conjecture \eqref{EqConjectureGonekM}.

In order to gain more insight into the size of $M(x)$, one needs further information on negative moments of the derivative of the Riemann zeta function (see for example \cite[Theorem 14.27]{Ti}). Note that to define such sums one needs to assume that all zeros of the zeta function are simple, which we might do in view of our observation above.  For a real number $k$ we define 
$$ J_{-k}(T):= \sum_{0<\gamma\leq T} \frac{1}{|\zeta'(\rho)|^{2k}}.$$
Gonek \cite{Go} conjectured that $J_{-1}(T)\ll T$ and showed that this conjecture together with RH implies the improved bound
$$M(x) \ll \sqrt{x} (\log x)^{3/2},$$
see \cite[Theorem 1 (i)]{Ng}. Under the same assumptions, 
Ng \cite{Ng} showed that $M(x)/\sqrt{x}$ has a  logarithmic limiting distribution. He also obtained a similar result to Gallagher's bound \eqref{EqGallagher} for $M(x)$. More precisely, assuming RH and the bound $J_{-1}(T)\ll T$, Ng 
showed that 
$$M(x)\ll \sqrt{x} (\log\log x)^{3/2},$$
except on a set of finite logarithmic measure. 

The conjectured bound $J_{-1}(T)\ll T$ is a special case of a more general conjecture, formulated independently by Gonek \cite{Go} and Hejhal \cite{He}, which states that  
\begin{equation}\label{ConjectureGoHe}
J_{-k}(T) \asymp_k T (\log T)^{(k-1)^2},
\end{equation}
for all real numbers $k$. However, it is now believed that this conjecture only holds if $k\leq 3/2$. Using random matrix theory, Hughes,  Keating and O’Connell \cite{HKO} refined this conjecture to an asymptotic formula 
$$
J_{-k}(T) \sim \mathcal{G}_k \frac{T}{2\pi}\left(\log \frac{T}{2\pi}\right)^{(k-1)^2},
$$
for $k\leq 3/2$, where $\mathcal{G}_k$ is an explicit positive constant. Partial results towards the Gonek-Hejhal Conjecture \eqref{ConjectureGoHe} for $k\geq 0$ are only known in the case of the implicit lower bound. Indeed, assuming RH and the simplicity of zeros of the zeta function, Heap, Li and Zhao \cite{HLZ} proved that
\begin{equation}\label{EqHeapLiZ}
J_{-k}(T) \gg_k T (\log T)^{(k-1)^2},
\end{equation}
for all positive rationals $k$, and this was recently extended to all real numbers $k\geq 0$ by Gao and Zhao \cite{GaZh}. 

As a consequence of a more general result concerning certain sums over the zeros of the Riemann zeta function (see Theorem \ref{ThmMain} below), we prove the omega result \eqref{EqOmegaPredictedM}, assuming the Effective LI conjecture, the Gonek-Hejhal Conjecture \eqref{ConjectureGoHe} for $k=1/2$, and a weak form of this conjecture for $k=1$. 
\begin{cor}\label{CorMobius}
Assume the Effective LI conjecture, together with the bounds $J_{-1/2}(T)\ll T(\log T)^{1/4}$ and $J_{-1}(T)\ll T^{\theta}$ for some $1\leq \theta < 3-\sqrt{3}$. Then we have 
$$ M(x)=\Omega_{\pm} \left(\sqrt{x}(\log\log\log x)^{5/4}\right). $$
\end{cor}


\subsection{General sums over the zeros of the zeta function}

By the explicit formula for the $\psi$ function (see  \cite[page 109]{Da}), we have
$$ \psi(x)= x-\sum_{|\gamma|\leq Y} \frac{x^{\rho}}{\rho}+O(\log x),$$
for $Y>x^{3/2}.$
Therefore, if we assume RH and pair each zero with its complex conjugate we deduce that for all $2\leq x\leq X$ 
\begin{equation}\label{ExplicitFormulaPsi}
\frac{\psi(x)-x}{\sqrt{x}}= - 2 \re\left(\sum_{0<\gamma\leq X^2} \frac{x^{i\gamma}}{\rho}\right)+O\left(\frac{\log x}{\sqrt{x}}\right).
\end{equation}

An analogous formula for the function $M(x)$ exists under the assumptions of RH and the conjectural bound $J_{-1}(T)\ll T^{\theta}$ for some $1\leq \theta < 3-\sqrt{3}$. Indeed, Akbary, Ng and Shahabi \cite[Eq. (4.18)]{ANS} showed that under these assumptions one has 
\begin{equation}\label{ExplicitFormulaM}
\frac{M(x)}{\sqrt{x}}=  2\re\left(\sum_{0<\gamma\leq X^2}\frac{x^{i\gamma}}{\rho\zeta'(\rho)}\right)+O(1),
\end{equation}
for all $2\leq x\leq X$.


We shall consider more general sums over the zeros, which are of the form 
$$ \Phi_{X, \mathbf{r}}(x)=\Phi_X(x):= \re\left(\sum_{0<\gamma \leq X} x^{i\gamma} r_{\gamma}\right),$$
where $\{r_{\gamma}\}_{\gamma>0}$ is a sequence of complex numbers satisfying the following assumptions:  
  \textbf{Assumption 1}:
  There exists a positive constant $A$ such that 
$$ 
H(T):=\sum_{0<\gamma\leq T} |r_{\gamma}|\asymp (\log T)^A.
$$ 
 
  \noindent \textbf{Assumption 2}:
Let $A$ be the constant in Assumption 1. As $T\to \infty$ we have 
 $$
L(T):= \sum_{0<\gamma\leq T} \gamma |r_{\gamma}|=o\big(T(\log T)^A\big).
$$
 
 \noindent \textbf{Assumption 3}: There exists a constant  $1\leq \theta <3-\sqrt{3}$ such that 
$$
\sum_{0<\gamma\leq T} \gamma^2 |r_{\gamma}|^2 \ll T^{\theta}.
$$
Our main result is the following theorem.

 \begin{thm}\label{ThmMain}
Assume the Effective LI conjecture. Let $\{r_{\gamma}\}_{\gamma>0}$ be  a sequence of complex numbers satisfying Assumptions 1-3. Let $X$ be large. There exist positive constants $C_1$ and $C_2$ such that 
\begin{equation}\label{MaxOmega}\max_{x\in [2, X]} \Phi_{X^2}(x) \geq C_1(\log\log\log X)^A.
\end{equation}
and 
\begin{equation}\label{MinOmega}\min_{x\in [2, X]} \Phi_{X^2}(x) \leq -C_2(\log\log\log X)^A.
\end{equation}
 \end{thm}

\textbf{Acknowledgements}. We  would like to thank Greg Martin and Nathan Ng for fruitful discussions and for their comments on an earlier draft of this paper. The author is supported by a junior chair of the Institut Universitaire de France. Part of this work was completed while the author was on a D\'el\'egation CNRS at the IRL3457 CRM-CNRS
in Montr\'eal. He would like to thank the CNRS for its support and the Centre de
Recherches Math\'ematiques for its excellent working conditions. 


\section{Key ingredients and outline of the Proof of Theorem \ref{ThmMain}}
 Let $\beta_{\gamma}=\arg r_{\gamma}$. For a real number $T\geq 2$ we define
\begin{equation}\label{DefShortSumZeros}
F(t, T):= \Phi_{T}(e^t)=\sum_{0<\gamma\leq T} \cos(\gamma t+\beta_{\gamma})  |r_{\gamma}|, 
 \end{equation}
where we changed variables from $x$ to $e^t$. Thus we need to consider the maximal and minimal values of $F(t, e^{2X})$ for $t\in [1, X]$. 
 

The proof of Theorem \ref{ThmMain} consists in three different steps. Our first goal is to show that the sum over the small zeros (with $\gamma \leq (\log X)^{1-\varepsilon}$) can be very large for many points $t\in [1, X]$. To this end we use the Effective LI conjecture to prove the following result.
\begin{pro}\label{ProDistributionSumZeros0}
 Assume the Effective LI conjecture. Let $\{r_{\gamma}\}_{\gamma>0}$ be  a sequence of complex numbers satisfying Assumption 1. Let $\varepsilon>0$ be small and fixed.  There exist positive constants $c_1$ and $c_2$ such that for $X$ large and  all real numbers $2\leq T\leq (\log X)^{1-\varepsilon}$ we have
\begin{equation}\label{EqMeasureLargeF} 
\frac{1}{X}\textup{meas} \{t\in [1, X] : F(t, T)> c_1 (\log T)^A\} \gg e^{-c_2 T\log T} ,
\end{equation}
 where meas denotes the Lebesgue measure on $\mathbb{R}$. Furthermore, we also have   
 \begin{equation}\label{EqMeasureSmallF} \frac{1}{X}\textup{meas} \{t\in [1, X] : F(t, T)<- c_1 (\log T)^A\}\gg e^{-c_2 T\log T}.
 \end{equation}
\end{pro}
In order to prove Theorem \ref{ThmMain} we need to shorten the sum $F(t, e^{2X})$ so that, roughly speaking,  we only keep the contribution of the very small zeros $\gamma\leq (\log X)^{1-\varepsilon}$. We do this in two steps. First, we achieve this for shorter sums $F(t, Y)$,  with $Y$ a power of $X$, using the Fejer kernel  
$$K(u):= \left(\frac{\sin(\pi u)}{\pi u}\right)^2.$$
Recall that
\begin{equation}\label{Fejer}\int_{-\infty}^{\infty} K(u) e^{2\pi i t u} du= \max(0, 1-|t|)=: k(t).
\end{equation}
We prove the following key lemma. 
\begin{lem}\label{LemFejer}
Let $\{r_{\gamma}\}_{\gamma>0}$ be a sequence of complex numbers satisfying Assumption 1. Let $Y\geq T\geq 2$ be real numbers. For any real number $Z\geq (\log Y)^A$ we have 
\begin{equation}\label{FejerMain}
F(t, T) = \int_{-Z}^{Z} \frac{T}{2\pi} K\left(\frac{Tu}{2\pi}\right) F(t+u, Y) du +O\left(\frac{L(T)}{T} \right).
\end{equation}
\end{lem}
If we take $Y= e^{2X}$ in this result, then we must have $Z \geq X^A$ which causes the shifts in $F(t+u, Y)$ to be very large if $A\geq 1$ (which is the case for the error term in the prime number theorem and the summatory function of the M\"obius function). To overcome this problem, we show that the difference between the integral on the right hand side of \eqref{FejerMain} with $Y=X$ and the same integral with $Y=e^{2X}$ is small on average over $t\in [1, X]$. To this end we define 
$$\ep(t, T_1, T_2):= \sum_{T_1 <\gamma \leq T_2} e^{i\gamma t} r_{\gamma}.$$
Our third and last step of the argument is to prove the following lemma.
\begin{lem}\label{LemSecondMomentError}
Let $\{r_{\gamma}\}_{\gamma>0}$ be a sequence of complex numbers satisfying Assumption 3. Let $X, Y\geq 2$ be real numbers.  For any positive real numbers $h, Z$ we have 
 $$ \frac{1}{X}\int_1^{X} \left|\int_{-Z}^Z h K(hu) \ep(t+u, Y, e^{2X}) du\right|^2dt \ll Y^{-\alpha},$$
 where the implicit constant is absolute. 
 \end{lem}
 We now show how to deduce Theorem \ref{ThmMain} from Proposition \ref{ProDistributionSumZeros0} and Lemmas \ref{LemFejer} and \ref{LemSecondMomentError}. 
 \begin{proof}[Proof of Theorem \ref{ThmMain}]
 Let $Z=(\log X)^A$ and define $\mathcal{E}(X)$ to be the set of points $t\in  [\sqrt{X}, X-\sqrt{X}]$  such that 
$$ \left|\int_{-Z}^Z \frac{T}{2\pi} K\left(\frac{Tu}{2\pi}\right) 
\big(F(t+u, e^{2X})-F(t+u, X)\big) du\right|>1.$$
Then, it follows from Lemma \ref{LemSecondMomentError} that 
\begin{equation}\label{MeasureEpsilon}
\begin{aligned}\text{meas}(\mathcal{E}(X)) &\leq \int_1^X \left|\int_{-Z}^Z \frac{T}{2\pi} K\left(\frac{Tu}{2\pi}\right) \big(F(t+u, e^{2X})-F(t+u, X)\big) du\right|^2dt\\
& \leq \int_1^X \left|\int_{-Z}^Z \frac{T}{2\pi} K\left(\frac{Tu}{2\pi}\right) \ep(t+u, X, e^{2X}) du\right|^2dt \ll X^{1-\alpha}.
\end{aligned}
\end{equation}
Let  $T=(\log X)^{1/2}$. Let $c_1$ be the constant in Proposition \ref{ProDistributionSumZeros0} and $\mathcal{A}(X)$ be the set of points $t\in  [\sqrt{X}, X-\sqrt{X}]$ such that $F(t, T)\geq c_1 (\log T)^A$.  Combining the estimate \eqref{MeasureEpsilon} with Proposition \ref{ProDistributionSumZeros0} we get 
$$
\text{meas}\big(\mathcal{A}(X)\setminus \mathcal{E}(X)\big) \geq X \exp\left(-\frac{\log X}{\log\log X}\right).
$$
We now let $t\in \mathcal{A}(X)\setminus \mathcal{E}(X)$. Then it follows from Lemma \ref{LemFejer}, Assumption 2 and our definition of the set $\mathcal{A}(X)$  that 
\begin{equation}\label{LBIntFejerF}
\int_{-Z}^{Z} \frac{T}{2\pi} K\left(\frac{Tu}{2\pi}\right) F(t+u, X) du\geq \frac{c_1}{2} (\log T)^A.
\end{equation}
Moreover, since $t\notin \mathcal{E}(X)$ then 
\begin{equation}\label{LBIntFejerF2}
\begin{aligned}
\int_{-Z}^{Z} \frac{T}{2\pi} K\left(\frac{Tu}{2\pi}\right) F(t+u, e^{2X}) du 
& \geq \int_{-Z}^{Z} \frac{T}{2\pi} K\left(\frac{Tu}{2\pi}\right) F(t+u, X) du -1\\
& \geq \frac{c_1}{2} (\log T)^A-1,
\end{aligned}
\end{equation}
by \eqref{LBIntFejerF}. On the other hand since $K(v)$ is positive for all $v$, we get 
\begin{equation}\label{BoundFejerIntegral}
\begin{aligned}\int_{-Z}^{Z} \frac{T}{2\pi} K\left(\frac{Tu}{2\pi}\right) F(t+u, e^{2X}) du\leq \max_{|u|\leq Z} F(t+u, e^{2X})\int_{-Z}^{Z} \frac{T}{2\pi} K\left(\frac{Tu}{2\pi}\right)du \leq \max_{|u|\leq Z} F(t+u, e^{2X})
\end{aligned}
\end{equation}
 since  
\begin{equation}\label{LaplaceFejer0}
\int_{-Z}^{Z} h K\left(hu\right)du
\leq  \int_{-\infty}^{\infty} hK\left(hu\right)du 
= \int_{-\infty}^{\infty}  K\left(v\right)du= k(0)= 1,
\end{equation}
for every positive real number $h$, by using the change of variable $v= hu$. Combining the bounds \eqref{LBIntFejerF2} and \eqref{BoundFejerIntegral} implies the existence of $y\in [1, X]$ such that 
$$ F(y, e^{2X}) \gg (\log\log X)^A.$$
This establishes \eqref{MaxOmega}.
To prove \eqref{MinOmega} 
we let $\mathcal{A}'(X)$ be the set of points $t\in [\sqrt{X}, X-\sqrt{X}]$ such $F(t, T)\leq -c_1 (\log T)^A$.
Now, if $t\in \mathcal{A}'(X)\setminus{\mathcal{E}(X)}$ then similarly to \eqref{LBIntFejerF2} one has 
$$
\int_{-Z}^{Z} \frac{T}{2\pi} K\left(\frac{Tu}{2\pi}\right) F(t+u, e^{2X}) du \leq -\frac{c_1}{2} (\log T)^A+1.
$$
To conclude we observe that
\begin{equation}\label{TruncFejer}
\begin{aligned}
\int_{-Z}^{Z} \frac{T}{2\pi} K\left(\frac{Tu}{2\pi}\right) F(t+u, e^{2X}) du &\geq \min_{|u|\leq Z} F(t+u, e^{2X})\int_{-Z}^{Z} \frac{T}{2\pi} K\left(\frac{Tu}{2\pi}\right)du \\
& \geq \min_{|u|\leq Z} F(t+u, e^{2X}) \left(1+O\left(\frac{1}{(\log X)^A}\right)\right),
\end{aligned}
\end{equation}
which follows from the estimate
$$
\int_{-Z}^{Z} \frac{T}{2\pi} K\left(\frac{Tu}{2\pi}\right)du
= \int_{-\infty}^{\infty} \frac{T}{2\pi} K\left(\frac{Tu}{2\pi}\right)du + O\left(\frac{1}{TZ}\right)=  1+ O\left(\frac{1}{ZT}\right),
$$
since $K(u)\leq 1/(\pi u)^2$ for $u\neq 0. $
 \end{proof}
We end this section by deducing Corollary \ref{CorMobius} for the summatory function of the M\"obius function. 
\begin{proof}[Proof of Corollary \ref{CorMobius}]
First, in view of \eqref{EqOmegaRHFalseM} and \eqref{EqOmegaMultiplicity} we might assume RH and the simplicity of zeros of the zeta function. Moreover, by the explicit formula \eqref{ExplicitFormulaM} 
and Theorem \ref{ThmMain} one only needs to verify that the sequence $\{r_{\gamma}\}_{\gamma>0}=\{(\rho \zeta'(\rho))^{-1}\}_{\gamma>0}$ verifies Assumptions 1, 2 and 3, with parameter $A=5/4.$ By \eqref{EqHeapLiZ} and our assumption we have
$$J_{-1/2}(T)=\sum_{0<\gamma\leq T} \frac{1}{|\zeta'(\rho)|}\asymp T(\log T)^{1/4}.$$
Hence, by partial summation we deduce that 
$$ \sum_{0<\gamma\leq T} \frac{1}{|\rho\zeta'(\rho)|}\asymp (\log T)^{5/4},  $$
which corresponds to Assumption 1 with $A=5/4$. Moreover, one has 
$$ \sum_{0<\gamma\leq T} \frac{\gamma}{|\rho\zeta'(\rho)|}\leq J_{-1/2}(T) \ll T (\log T)^{1/4}, $$
implying Assumption 2. Finally, Assumption 3 follows from the bound $J_{-1}(T)\ll T^{\theta}$.

\end{proof}
\section{The contribution of small zeros : Proof of Proposition \ref{ProDistributionSumZeros0}}

In order to detect large values of $F(t, T)$, a classical approach consists in estimating its large moments. However, it turns out that a more efficient way is to use the moment generating function (or Laplace transform) of $F(t, T)$. The reason is that, heuristically, the distribution of $F(t, T)$ should be close to that of the sum of random variables $\sum_{0<\gamma\leq T} |r_{\gamma}|\cos(\theta_{\gamma})$, where $\{\theta_{\gamma}\}_{\gamma>0}$ is a sequence of I. I. D. random variables indexed by the $\gamma$'s and  uniformly distributed on $[-\pi, \pi]$. Furthermore, it is easier to estimate the moment generating function of this sum rather than its moments, since the former  equals the product over the $\gamma$'s of the moment generating functions of $\cos(\theta_{\gamma})|r_{\gamma}|$, by the independence of the random variables $\{\theta_{\gamma}\}_{\gamma>0}$.

Using the Effective LI conjecture, we prove that as $t$ varies in $[1, X]$, the moment generating function of $F(t, T)$ is close to that of its corresponding probabilistic random model in a large range, if $T\leq (\log X)^{1-\varepsilon}$.  
\begin{pro}\label{ProLaplaceRandom}
Assume the Effective LI conjecture. Let $\{r_{\gamma}\}_{\gamma>0}$ be a sequence of complex numbers satisfying Assumption 1. Let $\varepsilon>0$ be small and fixed. Let $X$ be large and $2\leq T\leq (\log X)^{1-\varepsilon}$ be a real number. There exists a positive constant $c_0$ such that for all real numbers $s$ with $|s|\leq c_0 T(\log T)^{1-A} $ we have 
$$ 
\frac{1}{X}\int_2^X \exp\big(s F(t, T) \big)dt= \ex\left(\exp\left(s\sum_{0<\gamma \leq T}  |r_{\gamma}|\cos(\theta_{\gamma})\right)\right)+ O\left(e^{-N(T)}\right).
$$
\end{pro}

To prove this result, we show that the Effective LI conjecture implies that the sequence $\{\cos(\gamma  t+ \beta_{\gamma})\}_{0<\gamma\leq T}$ behaves like a sequence of independent random variables as $t$ varies in $[1, X]$, if $T\leq (\log X)^{1-\varepsilon}.$
\begin{lem}\label{LemKeyMomentsSinRand}
Assume the Effective LI conjecture. Let $\varepsilon>0$ be small and fixed. Let $X$ be large and $2\leq T\leq (\log X)^{1-\varepsilon}$ be a real number. Let $k\leq N(T)$ be a positive integer and $0<\gamma_1, \dots, \gamma_k<T$ be imaginary parts of zeros of $\zeta(s)$ (not necessarily distinct). For any real numbers $\beta_1, \dots, \beta_k$ we have 
$$ \frac1X\int_{1}^X\cos(\gamma_1 t+\beta_1)\cos(\gamma_2 t+\beta_2)\cdots \cos(\gamma_k t+\beta_k) dt= \ex\left(\cos(\theta_{\gamma_1}) \cdots\cos(\theta_{\gamma_k})\right) + O\left(X^{-2/3}\right). 
$$
\end{lem}
\begin{proof}
We have
\begin{equation}\label{MomentCos}
\begin{aligned}
&\frac{1}{X}\int_{1}^X\cos(\gamma_1 t+\beta_1)\cos(\gamma_2 t+\beta_2)\cdots \cos(\gamma_k t+\beta_k) dt \\
&= \frac{1}{2^k X}
\int_1^X \left(e^{i\gamma_1 t+i\beta_1}+e^{-i\gamma_1 t-i\beta_1}\right)\cdots \left(e^{i\gamma_k t+i\beta_k}+e^{-i\gamma_k t-i\beta_k}\right) dt \\
& = \frac{1}{2^k} \sum_{\delta_1, \delta_2, \dots, \delta_k \in \{-1, 1\}} e^{i(\delta_1\beta_1+ \cdots + \delta_k \beta_k)} \frac{1}{X}\int_{1}^X  e^{i(\delta_1 \gamma_1+ \cdots + \delta_k \gamma_k)t}dt.
\end{aligned}
\end{equation}
For $\delta_1, \delta_2, \dots, \delta_k \in \{-1, 1\}$, we group the $\gamma_j$ into distinct terms, and write $$\delta_1\gamma_1+\dots+ \delta_k \gamma_k=\ell_{1} \gamma_{\ell_1}+ \cdots + \ell_{J} \gamma_{\ell_{J}}$$ where $1\leq J\leq k$, the $\{\gamma_{\ell_j}\}_{1\le j\leq J}$ are distinct and the $\ell_j$ are integers such that $|\ell_j|\leq k\leq N(T)$  for all $j$.
First, we shall use the Effective LI conjecture to handle the contribution of the ``off-diagonal'' terms 
$\delta_1 \gamma_1+ \cdots + \delta_k \gamma_k\neq 0.$ Indeed, in this case we have
$$\int_1^X  e^{i(\delta_1 \gamma_1+ \cdots + \delta_k \gamma_k)t}dt= \int_1^X  e^{i(\ell_{1} \gamma_{\ell_1}+ \cdots + \ell_{J} \gamma_{\ell_{J}})t}dt \ll \frac{1}{|\ell_{1} \gamma_{\ell_1}+ \cdots + \ell_{J} \gamma_{\ell_{J}}|}\ll_{\varepsilon} e^{T ^{1+\varepsilon}}\ll X^{1/3},$$
by our assumption on $T$. Inserting this estimate in \eqref{MomentCos} we get 
\begin{equation}\label{MomentCos2}
\begin{aligned}
& \frac{1}{X}\int_1^X\cos(\gamma_1 t+\beta_1)\cos(\gamma_2 t+\beta_2)\cdots \cos(\gamma_k t+\beta_k) dt\\
& =  \frac{1}{2^k} \sum_{\substack{\delta_1, \delta_2, \dots, \delta_k \in \{-1, 1\}\\\delta_1 \gamma_1+ \cdots + \delta_k \gamma_k=0 }} e^{i(\delta_1\beta_1+ \cdots + \delta_k \beta_k)} +O\left(X^{-2/3}\right).
\end{aligned}
\end{equation}
We  now show that the main term on the right hand side of this asymptotic formula equals the corresponding moment of the random model. First, note that $\theta_{\gamma}+\beta_{\gamma}$ has the same distribution as $\theta_{\gamma}$. Therefore, similarly to \eqref{MomentCos} we have 
\begin{equation}\label{MomentsRandom}
\begin{aligned}
\ex\left(\cos(\theta_{\gamma_1}) \cdots\cos(\theta_{\gamma_k})\right)&=
\ex\left(\cos(\theta_{\gamma_1}+\beta_1) \cdots\cos(\theta_{\gamma_k}+\beta_k)\right)\\
&= \frac{1}{2^k}
\ex\Big(\left(e^{i\theta_{\gamma_1}+i\beta_1}+e^{-i\theta_{\gamma_1}-i\beta_1}\right)\cdots \left(e^{i\theta_{\gamma_k}+i\beta_k}+e^{-i\theta_{\gamma_k}-i\beta_k}\right)\Big)\\
&= \frac{1}{2^k} \sum_{\delta_1, \delta_2, \dots, \delta_k \in \{-1, 1\}} e^{i(\delta_1\beta_1+ \cdots + \delta_k \beta_k)}  \ex\left(e^{i(\delta_1 \theta_{\gamma_1}+ \cdots + \delta_k \theta_{\gamma_k})}\right).
\end{aligned}
\end{equation}
It remains to show that the last expectation is $0$ unless $\delta_1\gamma_1+\dots+ \delta_k \gamma_k=0,$ in which case it equals $1$. First assume that $\delta_1\gamma_1+\cdots+ \delta_k \gamma_k=0.$ This implies $\ell_{1} \gamma_{\ell_1}+ \cdots + \ell_{J} \gamma_{\ell_{J}}=0$ and hence $\ell_1=\dots=\ell_J=0$ since the $\{\gamma_{\ell_j}\}_{1\leq j\leq J}$ are linearly independent over the rationals. Therefore, we deduce that $\delta_1 \theta_{\gamma_1}+ \cdots + \delta_k \theta_{\gamma_k}= \ell_1 \theta_{\gamma_{\ell_1}}+ \cdots+ \ell_J \theta_{\gamma_{\ell_J}}=0$, and in particular $\ex\left(e^{i(\delta_1 \theta_{\gamma_1}+ \cdots + \delta_k \theta_{\gamma_k})}\right)=1.$ We now assume that  $\delta_1\gamma_1+\dots+ \delta_k \gamma_k=\ell_{1} \gamma_{\ell_1}+ \cdots + \ell_{J} \gamma_{\ell_{J}}\neq 0.$ Then there exists $1\leq j\leq J$ such that $\ell_j\neq 0$. In this case, by the independence of the $\{\theta_{\gamma}\}_{\gamma}$ we obtain 
\begin{align*}
\ex\left(e^{i(\delta_1 \theta_{\gamma_1}+ \cdots + \delta_k \theta_{\gamma_k})}\right)&= \ex\left(e^{i(\ell_1 \theta_{\gamma_{\ell_1}}+ \cdots+ \ell_J \theta_{\gamma_{\ell_J}})}\right)\\
&= \ex\left(e^{i\ell_j \theta_{\gamma_{\ell_j}}}\right)\ex\left(e^{i(\ell_1 \theta_{\gamma_{\ell_1}}+ \cdots+ \ell_{j-1} \theta_{\gamma_{\ell_{j-1}}}+ \ell_{j+1} \theta_{\gamma_{\ell_{j+1}}}+ \cdots + \ell_J \theta_{\gamma_{\ell_J}})}\right)=0,
\end{align*}
since $\ex\left(e^{i\ell_j \theta_{\gamma_{\ell_j}}}\right)=0.$ This completes the proof.
\end{proof}


We now prove Proposition \ref{ProLaplaceRandom}.
\begin{proof}[Proof of Proposition \ref{ProLaplaceRandom}] Let $R=N(T)$. First, note that 
$$
\exp\big(s F(t, T) \big)= \sum_{k=0}^R \frac{s^k}{k!}F(t, T)^k + E_1,
$$
where 
$$ E_1 \ll \sum_{k>R} \frac{(sH(T))^k}{k!}\ll \sum_{k>R} \left(\frac{3sH(T) }{k}\right)^k \leq \sum_{k>R} \left(\frac{3 sH(T)}{R}\right)^k \leq e^{-R},$$
by Assumption 1 and Stirling's formula, if $c_0$ is suitably small. Therefore, we deduce that 
\begin{equation}\label{ApproxLaplaceMoments}
\frac{1}{X}\int_1^X \exp\big(s F(t, T) \big)dt= \sum_{k=0}^R \frac{s^k}{k!} \frac{1}{X}\int_1^X F(t, T)^k dt + O(e^{-R}).
\end{equation}
Let $1\leq k\leq R$ be a positive integer. Expanding the inner integrand in the right hand side of \eqref{ApproxLaplaceMoments} and using Lemma \ref{LemKeyMomentsSinRand} we obtain
\begin{equation}\label{ExpandMomentGamma}
\begin{aligned}
&\frac{1}{X}\int_1^X F(t, T)^k dt\\
&= \sum_{0<\gamma_1,\dots, \gamma_k\leq T} |r_{\gamma_1}r_{\gamma_2} \cdots r_{\gamma_k}|\frac{1}{X}\int_1^X\cos(\gamma_1 t+\beta_{\gamma_1})\cos(\gamma_2 t+\beta_{\gamma_2})\cdots \cos(\gamma_k t+\beta_{\gamma_k}) dt\\
& = \sum_{0<\gamma_1,\dots, \gamma_k\leq T} |r_{\gamma_1}r_{\gamma_2} \cdots r_{\gamma_k}|\ex\left(\cos(\theta_{\gamma_1}) \cdots\cos(\theta_{\gamma_k})\right) 
+ O\left(X^{-2/3} H(T)^k\right)\\
& = \ex\left(\left(\sum_{0<\gamma \leq T}|r_{\gamma}|\cos(\theta_{\gamma})\right)^k\right) + O\left(X^{-1/2} \right),
\end{aligned}
\end{equation}
by Assumption 1. On the other hand, similarly to \eqref{ApproxLaplaceMoments} one has 
$$\ex\left(\exp\left(s \sum_{0<\gamma \leq T}|r_{\gamma}|\cos(\theta_{\gamma})  \right)\right)= \sum_{k=0}^R \frac{s^k}{k!}\ex\left(\left(\sum_{0<\gamma\leq T}  |r_{\gamma}|\cos(\theta_{\gamma})\right)^k\right) + O(e^{-R}).$$
Combining this estimate with \eqref{ApproxLaplaceMoments} and \eqref{ExpandMomentGamma} we deduce that 
\begin{align*}
\frac{1}{X}\int_2^X \exp\big(s F(t, T) \big)dt 
& = \sum_{k=0}^R \frac{s^k}{k!} \ex\left(\left(\sum_{0<\gamma \leq T}  |r_{\gamma}|\cos(\theta_{\gamma})\right)^k\right) + O\left( e^{s}X^{-1/2}+ e^{-R}\right) \\
& = \ex\left(\exp\left(s\sum_{0<\gamma \leq T} |r_{\gamma}|\cos(\theta_{\gamma})\right)\right)+ O(e^{-R}),
\end{align*}
as desired.
\end{proof}


We end this section by proving Proposition \ref{ProDistributionSumZeros0}. 

\begin{proof}[Proof of Proposition \ref{ProDistributionSumZeros0}]
We start by proving \eqref{EqMeasureLargeF}. Let $s=c_0 T(\log T)^{1-A}$, where $c_0$ is the constant from Proposition \ref{ProLaplaceRandom}. First, we shall establish a lower bound for the moment generating function
$$ \ex\left(\exp\left(s\sum_{0<\gamma \leq T} |r_{\gamma}|\cos(\theta_{\gamma})\right)\right).$$
By the independence of the $\theta_{\gamma}$ we have 
$$ \ex\left(\exp\left(s\sum_{0<\gamma \leq T}  |r_{\gamma}|\cos(\theta_{\gamma})\right)\right)= \prod_{0<\gamma \leq T}\ex\Big(\exp\big(s |r_{\gamma}|\cos(\theta_{\gamma})\big)\Big)=\prod_{0<\gamma \leq T} I_0(s|r_{\gamma}|), $$
where $I_0(t)=\sum_{n=0}^{\infty} t^{2k}/(2^k k!)^2$ is the modified Bessel function of the first kind. Note that $I_0(t)\geq 1$ for all $t\in \mathbb{R}$. Moreover, for $t\geq 0$ we have 
\begin{equation}\label{EstimateBessel}
I_0(t)= \frac{1}{2\pi}\int_{-\pi}^{\pi}e^{t\cos u} du\geq \frac{1}{2\pi} \int_0^{\pi/3} e^{t\cos u} du \geq \frac{e^{t/2}}{6}.
\end{equation}
Combining these estimates we obtain 
\begin{equation}\label{LowerBoundLaplace}
\ex\left(\exp\left(s\sum_{0<\gamma \leq T}  |r_{\gamma}|\cos(\theta_{\gamma})\right)\right) \geq \prod_{0<\gamma\leq \sqrt{T}} \frac{e^{s|r_{\gamma}|/2}}{6} \geq \exp\left(c_3 s(\log T)^A\right),
\end{equation}
for some positive constant $c_3$, by Assumption 1 and since $N(\sqrt{T})=o(s).$
We now use Proposition \ref{ProLaplaceRandom}, which implies 
\begin{equation}\label{LowerBoundLaplaceF}\frac{1}{X}\int_1^X \exp\big(s F(t, T) \big)dt \gg \exp\left(c_3 s(\log T)^A\right).
\end{equation}
Let $\mathcal{A}(X)$ be the set of points $t\in [1, X]$ such that 
$$F(t, T)\geq \frac{c_3}{2} (\log T)^A.$$ 
Then we have 
\begin{align*}
\int_1^X \exp\big(s F(t, T) \big)dt &= \int_{\mathcal{A}(X)} \exp\big(s F(t, T) \big)dt+ \int_{[1, X]\setminus\mathcal{A}(X)} \exp\big(s F(t, T) \big)dt\\
&= \int_{\mathcal{A}(X)}  \exp\big(s F(t, T) \big)dt +O\left(X\exp\left(\frac{c_3}{2}s (\log T)^A\right)\right).
\end{align*}
Hence, it follows from \eqref{LowerBoundLaplaceF} that 
$$\int_{\mathcal{A}(X)}  \exp\big(s F(t, T) \big)dt\gg 
X \exp\left(c_3 s(\log T)^A\right).$$
On the other hand, it follows from Assumption 1 that $F(t, T)\leq c_4 (\log T)^A$ for some positive constant $c_4$. Therefore, we deduce that 
\begin{equation}\label{LowerBoundMeasureA}\textup{meas}(\mathcal{A}(X))\geq \exp\left(-c_4 s(\log T)^A\right)\int_{\mathcal{A}(X)}  \exp\big(s F(t, T) \big)dt\gg 
X\exp\left(-c_4 s(\log T)^A\right).
\end{equation}
This yield the desired bound with $c_1= c_3/2$ and $c_2=c_0c_4$, by our choice of $s$. 

To prove \eqref{EqMeasureSmallF} we choose $s=-c_0 T(\log T)^{1-A}$ and proceed similarly to the proof of \eqref{LowerBoundLaplaceF} to show that 
$$ \frac{1}{X}\int_2^X \exp\big(sF(t, T))dt \gg \exp(-c_3s (\log T)^A),$$
by using that $I_0(t)$ is an even function. The result follows along the same lines of the proof of \eqref{LowerBoundMeasureA}.

\end{proof}

\section{Shortening the sum over the zeros: Proofs of Lemmas \ref{LemFejer} and \ref{LemSecondMomentError}}
In this section we establish Lemmas \ref{LemFejer} and \ref{LemSecondMomentError} using properties of the Fejer kernel. 

\begin{proof}[Proof of Lemma \ref{LemFejer}] First, we observe that 
$$ 
\int_{-\infty}^{\infty} \frac{T}{2\pi} K\left(\frac{Tu}{2\pi}\right) F(t+u, Y) du = \re\left(\sum_{0<\gamma\leq Y}    e^{i\gamma t+\beta_{\gamma}}|r_{\gamma}|  \int_{-\infty}^{\infty} \frac{T}{2\pi} K\left(\frac{Tu}{2\pi}\right) e^{i\gamma u} du\right).
$$
Now, using the change of variables $v= Tu/(2\pi)$ we get 
$$\int_{-\infty}^{\infty} \frac{T}{2\pi} K\left(\frac{Tu}{2\pi}\right) e^{i\gamma u} du= \int_{-\infty}^{\infty} K(v) \exp\left(2\pi i v \frac{\gamma}{T}\right)dv = \max\left(0, 1- \frac{\gamma}{T}\right),$$
by \eqref{Fejer}. Thus, we obtain 
\begin{equation}\label{FejerCut}
\int_{-\infty}^{\infty} \frac{T}{2\pi} K\left(\frac{Tu}{2\pi}\right) F(t+u, Y) du= \re\left(\sum_{0<\gamma \leq T} e^{i\gamma t+\beta_{\gamma}} |r_{\gamma}| \left(1- \frac{\gamma}{T}\right)\right)= F(t, T) + O\left(\frac{L(T)}{T}\right).
\end{equation}
Moreover, for any $Z>0$ we have 
$$ \int_{|u|>Z}\frac{T}{2\pi} K\left(\frac{Tu}{2\pi}\right) F(t+u, Y) du \ll T\int_{|u|>Z} \frac{H(Y)}{(Tu)^2}du \ll \frac{(\log Y)^A}{ZT},$$
by Assumption 1 since $|K(u)|\leq 1/(\pi u)^2$ for all $u\neq 0.$ Inserting this bound in \eqref{FejerCut} and using our assumption on $Z$ completes the proof. 
\end{proof}
To prove Lemma \ref{LemSecondMomentError} we need the following result, which is established by Akbary, Ng and Shahabi in \cite{ANS}.
\begin{lem}\label{LemmaAkbaryNg}
Let $\{r_{\gamma}\}_{\gamma>0}$ be a sequence of complex numbers satisfying Assumption 3. There exists a positive constant $\alpha$, that depends only on $\theta$, such that for all real numbers $X_2>X_1>1$ we have
$$\sum_{X_1<\gamma_1, \gamma_2\leq X_2} |r_{\gamma_1}r_{\gamma_2}|\min\left(1, \frac{1}{|\gamma_1-\gamma_2|}\right) \ll X_1^{-\alpha},$$
where the implicit constant is absolute.
\end{lem}
\begin{proof}
This follows from the proof of Theorem 1.2 b) of \cite{ANS}. The condition on the sequence $\{r_{\gamma}\}_{\gamma>0}$ corresponds to that in Corollary 1.3 b) of \cite{ANS} (which is a consequence of \cite[Theorem 1.2 b)]{ANS}). 
\end{proof}

 \begin{proof}[Proof of Lemma \ref{LemSecondMomentError}]
Expanding the integrand we obtain 
\begin{equation}
\label{ExpandAverageError}
\begin{aligned}
&\frac{1}{X}\int_1^X \left|\int_{-Z}^Z h K(hu) \ep(t+u, Y, e^{2X}) du\right|^2dt\\
&= \int_{-Z}^Z\int_{-Z}^Z h^2 K(hu) K(hv) \left(\frac{1}{X}\int_2^X \ep(t+u, Y, e^{2X}) \overline{\ep(t+v, Y, e^{2X})} dt \right) du dv.
\end{aligned}
\end{equation}
The inner integral equals
\begin{align*} 
&\frac{1}{X}\int_1^X \ep(t+u, Y, e^{2X}) \overline{\ep(t+v, Y, e^{2X})} dt \\&= \frac{1}{X}\int_1^X \sum_{Y<\gamma_1\leq e^{2X}} e^{i\gamma_1 (t+u)}r_{\gamma_{1}} \sum_{Y<\gamma_2\leq e^{2X}} e^{-i\gamma_2 (t+v)} r_{\gamma_2}  dt\\
&= \sum_{Y<\gamma_1, \gamma_2\leq e^{2X}}e^{i(\gamma_1 u-\gamma_2 v)} r_{\gamma_1}r_{\gamma_2} \frac{1}{X}\int_1^X e^{i(\gamma_1-\gamma_2) t}dt\\
& \ll \sum_{Y<\gamma_1, \gamma_2\leq e^{2X}} |r_{\gamma_1}r_{\gamma_2}|\min\left(1, \frac{1}{|\gamma_1-\gamma_2|}\right) \ll Y^{-\alpha}
\end{align*}
for some positive constant $\alpha$ by Lemma \ref{LemmaAkbaryNg}.
 Inserting this bound in \eqref{ExpandAverageError} gives 
$$
\frac{1}{X}\int_1^X \left|\int_{-Z}^Z h K(hu) \ep(t+u, Y, e^{2X}) du\right|^2dt \ll Y^{-\alpha} \left(\int_{-Z}^Z h K(hu) du\right)^2\ll Y^{-\alpha},
$$
by \eqref{LaplaceFejer0}.
This completes the proof.
 \end{proof}

\section{The error term in the prime number theorem: Proof of Theorem \ref{ThmErrorPNT}}
By the explicit formula \eqref{ExplicitFormulaPsi} we have
\begin{equation}\label{EqExplicitRHPNT}
\frac{\psi(e^t)-e^t}{e^{t/2}}= - 2\sum_{0<\gamma\leq e^{2X}} \frac{\sin(\gamma t)}{\gamma} +O(1),
\end{equation}
for all $1\leq t\leq X$, under the assumption of RH. In order to prove the inequalities \eqref{LimsupPNT} and \eqref{LiminfPNT} we need to establish a more precise version of Proposition \ref{ProDistributionSumZeros0} for the distribution of the extreme values of the sum 
$$F_p(t, T):=2\sum_{0<\gamma\leq T} \frac{\sin(\gamma t)}{\gamma},$$
for $2\leq T\leq (\log x)^{1-\varepsilon}.$ This can be achieved since the sequence $\{1/\gamma\}$ is decreasing and well behaved. In particular, we have the following proposition which we deduce from the results of \cite{GrLa}. 

\begin{pro}\label{ProDistributionSumZeros}
Assume the Effective LI conjecture. Let $X$ be large and $2\leq T\leq (\log X)^{1-\varepsilon}$ be a real number. There exists an absolute constant $C>0$ such that uniformly for $V$ in the range $1\leq V\leq (\log T- \log\log T -C)^2/(2\pi)$ we have 
$$ \frac{1}{X} \textup{meas} \{t\in [1, X] : F_p(t, T)> V\}= \exp\left(-e^{-\eta-1}\sqrt{2\pi V} \exp\left(\sqrt{2\pi V}\right)\left(1+O\left(\frac{\sqrt{\log V}}{V^{1/4}}\right)\right)\right), 
$$
where $$\eta =\int_{0}^1 \frac{\log I_0(u)}{u^2}du + \int_{1}^{\infty} \frac{\log I_0(u)-1}{u^2}du.$$
Moreover, the same result holds for $$ \frac{1}{X} \textup{meas} \{t\in [1, X] : F_p(t, T)< -V\}$$ in the same range of $V$.
\end{pro}
\begin{proof}
First, it follows from Lemma \ref{LemKeyMomentsSinRand} (with all the $\beta$'s being $-\pi/2$) that the sequence $\{\sin(\gamma t)\}_{0<\gamma\leq T}$ verifies the Approximate  Independence Hypothesis AIH$(L, Q)$ in \cite{GrLa} with parameters $Q=T$ and $L= N(T)$ for any real number $2\leq T\leq (\log X)^{1-\varepsilon}$, and where the sequence of associated independent random variables is the sequence $\{\cos(\theta_{\gamma})\}_{\gamma>0}$. This sequence of random variables satisfies the assumptions of Theorem 1.2 of \cite{GrLa}, and hence we can apply this result with parameters 
$A=2$, $\alpha=1/(2\pi)$, $\beta=2\pi$, and $W= (4\pi V+O(1))^{1/2}$. Therefore, it follows from Theorem 1.2 and Remark 1.3 of \cite{GrLa} that uniformly for $W\leq \log T- \log\log T- C_0 $ (for some positive constant $C_0$) we have 
$$\frac{1}{X} \textup{meas} \left\{t\in [1, X] : \sum_{0<\gamma\leq T} \frac{\sin(\gamma t)}{\gamma}> V\right\}= \exp\left(-e^{-\eta-1} W e^W\left(1+O\left(\frac{\sqrt{\log W}}{W^{1/2}}\right)\right)\right).$$
Replacing $W$ by $V$ implies the result. The proof for the measure of $t\in [1, X]$ such that $F(t, T)<-V$ follows along the same lines, by considering the sequence $\{\sin(-\gamma t)\}_{0<\gamma\leq T}$ instead. 
\end{proof}
We will now deduce Theorem \ref{ThmErrorPNT} along the same lines of the proof of Theorem \ref{ThmMain}.
\begin{proof}[Proof of Theorem \ref{ThmErrorPNT}]
First, we might assume RH since otherwise the result follows trivially by \eqref{EqRHFALSEPNT}. We only prove \eqref{LimsupPNT} since the proof of \eqref{LiminfPNT} is similar. By the asymptotic formula $N(T)\sim (T\log T)/(2\pi)$ and partial summation, one easily verifies that the sequence $\{1/\gamma\}_{\gamma>0}$ satisfies Assumptions 1, 2 and 3, with $A=2$ and any $\theta>1$. Hence, we may apply Lemmas \ref{LemFejer} and \ref{LemSecondMomentError} as in the proof of Theorem \ref{ThmMain}.

Let $T=(\log X)^{1-\varepsilon}$ and $\mathcal{B}(X)$ be the set of points $t\in [\sqrt{X}, X-\sqrt{X}]$ such that 
$$F_p(t, T)\leq -\frac{1-2\varepsilon}{2\pi } (\log\log X)^2.$$
Moreover, as in the proof of Theorem \ref{ThmMain}, we let $Z=(\log X)^2$ and define $\mathcal{E}(X)$ to be the set of $t\in  [\sqrt{X}, X-\sqrt{X}]$  such that 
$$ \left|\int_{-Z}^Z \frac{T}{2\pi} K\left(\frac{Tu}{2\pi}\right) 
\big(F_p(t+u, e^{2X})-F_p(t+u, X)\big) du\right|>1.$$
Then, it follows from Proposition \ref{ProDistributionSumZeros} together with the estimate \eqref{MeasureEpsilon} that 
$$ \textup{meas}\big(\mathcal{B}(X)\setminus{\mathcal{E}(X)}\big) \gg X\exp\left(-\frac{\log X}{\log\log X}\right).$$ 
Now, if $t\in \mathcal{B}(X)\setminus{\mathcal{E}(X)}$ then using Lemmas \ref{LemFejer} and \ref{LemSecondMomentError} we obtain similarly to \eqref{LBIntFejerF2} that
$$
\int_{-Z}^{Z} \frac{T}{2\pi} K\left(\frac{Tu}{2\pi}\right) F_p(t+u, e^{2X}) du \leq -\frac{1-2\varepsilon}{2\pi } (\log\log X)^2 +O(\log\log X),
$$
since in our case we have  $L(T)/T= N(T)/T\ll \log\log X$. 
For such $t$'s we deduce from \eqref{TruncFejer} that
$$ 
\min_{|u|\leq Z} F_p(t+u, e^{2X}) \leq -\frac{1-2\varepsilon}{2\pi } (\log\log X)^2 + O(\log\log X).
$$
Inserting this estimate in \eqref{EqExplicitRHPNT} and choosing $\varepsilon$ to be arbitrarily small completes the proof.
\end{proof}

\end{document}